\documentclass[12pt]{article}

\usepackage{amssymb,amsthm,amsmath}
\usepackage{latexsym}
\usepackage{graphicx}
\usepackage{xspace}
\usepackage{mathtools} 

\usepackage[a4paper,top=25mm,bottom=30mm]{geometry}

\usepackage{hyperref}

\newtheorem{theorem}{Theorem}[section]

\newtheorem{claim}[theorem]{Claim}

\newtheorem{cor}[theorem]{Corollary}
\newtheorem{obs}[theorem]{Observation}
\newtheorem{defi}[theorem]{Definition}

\newcommand{\cl}[1]{\mbox{\ensuremath{\mathbf{#1}}}\xspace}
\newcommand{\la}[1]{\mbox{\sc{#1}}\xspace}

\newcommand{\new}{^{\mathrm{new}}}
\newcommand{\bi}{^{\mathrm{bi}}}

\DeclareMathOperator{\Vol}{Vol}

\title{Search for the end of a path in the \\$d$-dimensional grid and
  in other graphs}

\author{D\'aniel Gerbner\thanks{%
{Hungarian Academy of Sciences, Alfr\'ed R\'enyi Institute of
  Mathematics, P.O.B.~127, Budapest H-1364, Hungary}.
 \nolinkurl
{gerbner.daniel@renyi.mta.hu}.
Supported by the Hungarian National Science Fund (OTKA), grant PD 109537}
 \and Bal\'azs Keszegh\thanks{%
{Hungarian Academy of Sciences, Alfr\'ed R\'enyi Institute of
  Mathematics, P.O.B.~127, Budapest H-1364, Hungary}.
 \nolinkurl
{keszegh@renyi.hu}.
 Supported by the Hungarian
Scientific Research Fund (OTKA), grant PD 108406, NN 102029 (EUROGIGA project GraDR 10-EuroGIGA-OP-003), NK 78439, by the J\'anos Bolyai Research Scholarship of the Hungarian Academy of Sciences, and DAAD.}
\and D\"om\"ot\"or P\'alv\"olgyi\thanks{%
{Institute of Mathematics, E\"otv\"os University, P\'azm\'any P\'eter s\'et\'any 1/C, H-1117, Budapest,
Hungary}.
 \nolinkurl
{domotorp@gmail.com}.
Supported by the
  Hungarian
Scientific Research Fund (OTKA), grant PD 104386 and NN
  102029 (EUROGIGA project GraDR 10-EuroGIGA-OP-003), and the J\'anos
  Bolyai Research Scholarship of the Hungarian Academy of Sciences.}
\and {G\"unter Rote}\thanks{%
{Freie Universit\"at Berlin, Institut f\"ur Informatik,
 Takustra\ss e~9, 14195 Berlin, Germany}. 
 \nolinkurl
 {rote@inf.fu-berlin.de}.
Supported by the
ESF EUROCORES programme EuroGIGA-VORONOI, Deutsche
Forschungsgemeinschaft (DFG): RO 2338/5-1.}
\and G\'abor Wiener\thanks{%
{Department of Computer Science and Information Theory, Budapest
University of Technology and Economics, M\H{u}egyetem rkp. 3., H-1111, Budapest, 
Hungary}.
 \nolinkurl
{wiener@cs.bme.hu}.
Supported by the Hungarian
Scientific Research Fund (OTKA), grant 108947, and the J\'anos
  Bolyai Research Scholarship of the Hungarian Academy of Sciences.}}

\begin{document}

\maketitle




\begin{abstract}
We consider the worst-case query complexity of some variants of certain \cl{PPAD}-complete search problems.
Suppose we are given a graph $G$ and a vertex $s \in V(G)$.
We denote the directed graph obtained from $G$ by directing all edges in both directions by $G'$.
$D$ is a directed subgraph of $G'$ which is unknown to us, except that it consists of vertex-disjoint directed paths and cycles and one of the paths originates in $s$.
Our goal is to find an endvertex of a path by using as few queries as
possible.
A query specifies a vertex $v\in V(G)$,
and the answer is the set of the edges of $D$ incident to $v$,
together with their directions.

We also show lower bounds for the special case when $D$ consists of a single path.
Our proofs use the theory of graph separators.
Finally, we consider the case when the graph $G$ is a grid graph.
 In this case, using the connection with separators, we give asymptotically tight bounds as a function of the size of the grid, if the dimension of the grid  is considered as fixed.
In order to do this, we prove a separator theorem about grid graphs, which is interesting on its own right.
\end{abstract}

\section{Introduction}
This paper deals with the following search problem.
 We are given a simple, undirected, connected graph $G$ and a vertex $s\in V(G)$.
We denote the directed graph obtained from $G$ by directing all edges in both directions by $G'$.
Let $D$ be a directed subgraph of $G'$, which is the vertex-disjoint union of a directed path starting at $s$ and possibly some other directed paths and cycles.
$D$ is unknown to us, and our goal is to identify an endvertex of a
directed path.
We may \emph{query} a vertex $v$, and
as an answer, we learn the edges of $D$ incident to $v$ together with their
directions.
In particular, if the answer is only one incoming edge, then we know
that $v$ is an endvertex.
We analyze the minimum number of queries that are necessary in the worst case.

We give lower bounds in the more restrictive model where we know $D$
is one directed path. Note that if instead of looking for an
endvertex, we look for an ending or a starting vertex of a path
(different from $s$), then this model still gives a lower bound for this easier problem.
In Section~\ref{sec:conclusion} we mention some additional models.



Denote by $h(G)$ the minimum number of queries
needed to find an endvertex in the worst case for any $s\in G$.
If we know that $D$ is one directed path, denote this quantity by $h_P(G)$.

\paragraph{Biseparators and multiseparators.}

To state some of our results we need to define separators of graphs.
This notion can be 
defined in two different ways and both definitions are widely used.
Here we distinguish between the two definitions.

\begin{defi}
  \begin{enumerate}
  \item 
Given a graph $G=(V,E)$, a subset $S\subseteq V$ is called an
\emph{$\alpha$-biseparator} of $G$ if $V\setminus S$ can be divided into two parts, $A$
and $B$, such that there are no edges between $A$ and $B$, and both have
cardinality at most $\alpha |V|$.
\item 
Given a graph $G=(V,E)$, a subset $S\subseteq V$ is called an
\emph{$\alpha$-multiseparator} of $G$ if every connected component of $V\setminus S$
has cardinality at most $\alpha |V|$.
  \end{enumerate}
\end{defi}

Note that $A$ or $B$ in the definition of a biseparator can be empty:
we do not require $V\setminus S$ to be disconnected. Small
biseparators make sense only for $\alpha\ge 1/2$.

Given these definitions, when we write \emph{separator}, it can mean either a biseparator or a multiseparator, as in many cases it makes no difference.
In the literature, the notation $f(n)$-separator can also be found,
where $f(n)$ is an upper bound on the cardinality of $S$ in terms of
the number $n$ of vertices.
In this paper it is more straightforward to fix $\alpha$ and then look for the smallest $\alpha$-separator.
Therefore, we let $s\bi _\alpha(G)$ be the minimum cardinality of an $\alpha$-biseparator in $G$ and $s^m_\alpha(G)$ be the minimum cardinality of an $\alpha$-multi\-se\-pa\-ra\-tor in $G$.

It follows from the definitions that every $\alpha$-biseparator is an
$\alpha$-multi\-se\-pa\-ra\-tor, and thus
 $s\bi _\alpha(G)\ge s^m_\alpha(G)$. In many cases they are of the
 same order of magnitude. In particular, if we have a bound
 $s^m _\alpha(G)\le O(n^c)$
 for a class of graphs
 which is closed under taking subgraphs
 for some $c<1$ and for \emph{some arbitrary} $\alpha<1$, we get
 the same asymptotic bound on
 $s\bi_{1/2}(G)$, by iteratively separating one of the components.
 However, there are cases when multiseparators are much smaller than biseparators.
For example, if $G$ consists of three disjoint cliques of equal size,
all connected to a degree-three vertex, then $s^m_{1/2}(G)=1$ but
$s\bi _{1/2}(G)=\lceil n/6\rceil $.
For any tree, $s^m_{1/2}(G)=1$ but it is not hard to show
 that for a complete ternary tree, $s\bi _{1/2}(G)=\Theta(\log n)$,
see Appendix~\ref{ternary}.
 Finally, if we consider a class of graphs closed under taking
 subgraphs, by repeatedly refining the separation, then it is obvious
 that $s^m_\alpha(G)$ and $s^m_{\alpha'}(G)$ have the same order of magnitude
for any two constants $\alpha$ and $\alpha'$.


\paragraph{Results.}

Our main result establishes a connection between
the biseparators and the search complexity for  general graphs.

\begin{theorem}\label{alsofeles}
For any connected graph $G$ with at least 2 vertices, we have $s\bi _{1/2}(G) \le h_P(G)\le h(G)$.
\end{theorem}

We can prove an upper bound of the same order of magnitude, if every subgraph has small multiseparators.
Note that when bounding $h(G)$, $s\bi (G)$, the larger of the separators, gives the lower bound and $s^{m}(G)$, the smaller one, gives the almost matching upper bound, which implies that indeed for a large class of graphs $s\bi (G)$ and $s^m(G)$ have the same order of magnitude.


\begin{theorem}\label{subhom}
Let $0<\alpha,\beta<1$ be constants,
let $f$ be a monotone
function,
and let $G$ be a graph such that any subgraph $H$ of $G$ has an
$\alpha$-multiseparator of size at most $f(|V(H)|)$.
If
 $f(\alpha x)\le \beta f(x)$ for all $x>0$,
then $$h_P(G)\le h(G)\le \frac{f(|V(G)|)}{1-\beta}.$$
\end{theorem}
The condition on $f$ could be interpreted as having ``at least polynomial
growth''.
The condition is fulfilled by the function $f(x)=\mathrm{const}\cdot
x^c
$ if and only if $c
 \ge \log_\alpha \beta$.
To put it differently, if $\alpha$ and $c>0$ are given, the
theorem applies with $\beta := \alpha^c$.



We also study the search problem for the special case of grid graphs.

\begin{defi}
Let $d$ be a positive integer and $(n_1,\ldots n_d)$ a sequence of positive integers.
The $d$-dimensional \emph{grid graph} of side length $(n_1,\ldots n_d)$, denoted by $G_d(n_1,\ldots n_d)$, has vertex set $\bigtimes_i \{0,1,2, \ldots , n_i-1\}$, and there is an edge between two vertices if and only if they differ in exactly one coordinate and the difference is $1$.
If $n_1=n_2=\dots=n_d$, then we simply write $G_d(n)$.
\end{defi}



We estimate the search complexity of grid graphs as follows.
\begin{theorem}\label{racs}
$\Omega (n^{d-1}/\sqrt{d})\le h_P(G_d(n))\le h(G_d(n))\le O(n^{d-1})$.
\end{theorem}

As a tool, we will prove a bound on the cardinality of separators of grid graphs, using classic results from the theory of vertex isoperimetric problems and cube slicing.

\begin{theorem}\label{parti} The smallest $1/2$-biseparator of the grid graph $G_d(n)$ has cardinality $s\bi(G_d(n)) = \Theta(n^{d-1}/\sqrt{d})$.
\end{theorem}

We note that when considering grid graphs, one could also study the
related problem that the path starting at $s$ is monotone, i.e., if $u$ and $v$ are on the path and $u\leq v$ (according to the usual partial order of the vectors), then the edge between $u$ and $v$ (if it exists) is directed towards $v$.
In this case the needed number of queries reduces dramatically.
Indeed, the trivial algorithm which follows the path uses at most $dn$ queries.
In two dimensions we could improve slightly this upper bound, yet there is a more
significant improvement by 
Xiaoming Sun (personal communication),
who proved that
$8n/5$ queries are enough in two dimensions. From below, at least
$n-2$ queries are needed regardless of $d$ \cite[Lemma 6]{hirsch}.
This problem resembles the
pyramid-path search problem (but it is not exactly the same),
where also a lower bound of $n$ is proved for the two-dimensional
case \cite{pypa}.

\paragraph{Motivation.}

Hirsch, Papadimitriou and Vavasis~\cite{hirsch} have proved worst-case lower bounds for finding Brouwer fixed points for algorithms using only function evaluation.
They showed a lower bound that is exponential in the dimension, disproving the conjecture that Scarf's algorithm is polynomial.
In our language, they have (implicitly) proved that $h(G_d(n))=\Omega(n^{d-2}/d^2)$ \cite[Lemma 16]{hirsch}.
Our
Theorem~\ref{racs} 
 is an improvement of their result, although we do not use the continuous setting but rather focus only on the discretization of the problem.

Later, Papadimitriou \cite{papa} considered similar complexity search problems in great detail and defined corresponding complexity classes \cl{PPA}, \cl{PPAD}, etc.
In his model, an exponential-size graph is given by a \emph{succinct} representation, i.e., by the description of a Turing-machine $T$.
The vertices of the graph correspond to binary sequences of length $n$ and if we input such a sequence to $T$, it outputs all the neighbors of the corresponding vertex in polynomial time (thus the degrees are bounded by a polynomial).
Therefore, in his model, instead of considering query cost, one can work with the classical running time of the algorithm that gets $T$ as input.
If the algorithm uses $T$ as a black box, we get back the query-cost model.

Papadimitriou considered the problem when the maximum degree of the graph is $2$, i.e., it consists of vertex disjoint paths and cycles and we are also given, as part of the input, a degree-one vertex, $s$, and our goal is to output another degree-one vertex.
This search problem is denoted by \la{LEAF}, and the complexity class
\cl{PPA}
 is defined such that \la{LEAF} is complete for
\cl{PPA}.
 (\cl{PPA} stands for ``Polynomial Parity Argument''.)

 Papadimitriou introduced another variant, where the underlying graph is directed ($T$ outputs
both the in- and out-neighbors of its input in this case), the in-
and out-degree of every vertex is at most one, and we are given a
starting vertex $s$ with in-degree zero and out-degree one.
Therefore, the resulting digraph is the vertex-disjoint union of a directed path starting at $s$ and possibly some other directed paths and cycles, exactly like in the problem that we study.
Here our goal can be either to output an in-degree one, out-degree zero vertex (called \la{LEAFDS} problem) or an in-degree plus out-degree equals one vertex (called \la{LEAFD} problem), which means the end of a path, just like in the problem we study.
Thus, the query-cost of \la{LEAFD} is exactly $h(K_{2^n})$.

The complexity classes for which the problems \la{LEAFDS} and \la{LEAFD} are complete are denoted, respectively, by \cl{PPADS} and \cl{PPAD}.
It is easy to see that \cl{PPAD} is contained in both \cl{PPA} and \cl{PPADS}, while an oracle separation is known for the two latter classes \cite{beame}.
Nowadays \cl{PPAD} enjoys huge popularity, as several problems, among
them finding an $\epsilon$-approximate Nash-equilibrium, turned out to
be \cl{PPAD}-com\-plete.
This is why this paper focuses on $h(G)$, the query-cost version of \cl{PPAD}, though most of our results would also hold for the other variants.

An extensive list of \cl{PPAD}-com\-plete problems can be found on Wikipedia.

\section{Upper bounds}

\begin{obs}\label{cut} Suppose that the connected components of $G\setminus S$ are $Y_1,\dots, Y_k$.
If every vertex of $S$ has been queried, we know a $Y_i$ which contains an endvertex (or that an endvertex is in $S$, hence already identified).
\end{obs}

\begin{proof} The answers clearly show how many times we enter and leave $S$ from each component $Y_i$.
If we enter a component $Y_i$ more times than we leave it, then $Y_i$ must contain an endvertex.
If there is no such component, the component containing $s$ must contain an endvertex.
\end{proof}

This simple observation is crucial for our upper bounds and it does
not hold if the answers would contain only the edges leaving the
queried vertex.

\begin{proof}[Proof of Theorem~\ref{subhom}.]
Let us choose an $\alpha$-multiseparator $S_1$ with $|S_1|\le
f(|V(G)|)$ which cuts $G$ into parts $Y_1, \dots, Y_k$, and query all
 vertices of $S_1$. By Observation \ref{cut} we know a part
$Y_j$ which contains an endvertex. Let $G_1$ be $G$ restricted to
$Y_j$ and choose an $\alpha$-multiseparator $S_2$ of size at most
$f(|V(G_1)|)$, which cuts $G_1$ into parts $Z_1, \dots, Z_l$.

Then $S_1\cup S_2$ is a separator of $G$, which cuts
it into parts $Y_1, \dots,Y_{j-1},\allowbreak Y_{j+1},\allowbreak \dots,\allowbreak Y_k,\allowbreak Z_1,\allowbreak
\dots, Z_l$. Thus, by again using Observation \ref{cut}
after asking every vertex of $S_1 \cup S_2$ we
know which part $Z_i$ contains an endvertex.

After this we can continue the same way, defining $G_2$ and asking
$S_3$, defining $G_3$ and asking $S_4$ and so on, until an endvertex is in some $S_i$.
As $|V(G_j)| \le \alpha |V(G_{j-1})|$ for any $j$, one can easily see
that $|V(G_j)| \le \alpha^j |V|$. 
By the assumptions on $f$,
 $f(|S_j|) \le f(|V(G_{j-1})|) \le
f(\alpha^{j-1}|V|) \le \beta^{j-1}f(|V|)$. Altogether at most
$\sum_{j=1}^\infty \beta^{j-1}f(|V|) \le f(|V|)/(1-\beta)$
questions were asked.
\end{proof}

A celebrated theorem of Lipton and Tarjan \cite{lt} states that
planar graphs have $2/3$-separators of size at most
$\sqrt{8}\cdot \sqrt{|V|}$.
Thus we have the following corollary.

\begin{cor} If $G$ is planar, then $h(G) = O(\sqrt{|V|})$.
\end{cor}

Now, let us look at $d$-dimensional grid graphs.
Miller, Teng and Vavasis \cite{mtv} introduced 
the so-called overlap graphs for every $d$ and proved that
every member $G$ of the class has separator of size
$O(|V(G)|^{(d-1)/d})$. They
mention that any subset of the $d$-dimensional infinite grid graph
belongs to the class of overlap graphs. The
polynomial function $f(x)=cx^{(d-1)/d}$ satisfies the assumption of
Theorem~\ref{subhom}.
Since
 $|V(G_d(n))|=n^d$,
this 
 implies that $h(G)=O(n^{d-1})$.
 Here we show that
the multiplicative constant is less than~3.

\begin{theorem}\label{th_cycle} $h(G_d(n))\le (2+\frac1{2^{d-1}-1})n^{d-1}$.
\end{theorem}

\begin{proof}
We follow the proof of Theorem \ref{subhom}, but the cuts we use
are always axis-aligned hyperplanes, which cut the current part into two
smaller grid graphs. More precisely, for any $i$
let $j \equiv i \bmod d$, $0\leq j \leq d-1$; now
$S_i$ is a hyperplane perpendicular to
the $j^{th}$ coordinate axis,
and it cuts $G_{i-1}$ into two parts of
size at most $|V(G_{i-1})|/2$. One can easily see that this is
possible and $|S_{i+1}|\le |S_i|/2$, except if $j=0$, in
which case $|S_{i+1}|\le |S_i|$. This means that there are at most
$$n^{d-1}(1+1/2+1/4+\ldots+1/2^{d-1})(1+1/2^{d-1}+1/2^{2(d-1)}+\ldots)$$
$$\le n^{d-1}(2-1/2^{d-1})\frac1{1-1/2^{d-1}}=n^{d-1}\left(2+\frac1{2^{d-1}-1}\right)$$
queries.
\end{proof}

\section{Lower bounds}\label{lower}
Before proving Theorem \ref{parti} which claims that any $1/2$-separator in the grid
graph $G_d(n)$ has cardinality $\Omega(n^{d-1}/\sqrt{d})$, we present a slightly weaker result, as it has a short proof not using results from the theory of isoperimetric problems.

\begin{claim}
Any $\alpha$-multiseparator in the grid graph $G_d(n)$ has cardinality at least $(1-\alpha) n^{d-1}/d$ for $\alpha\ge 1/2$.
\end{claim}

\begin{proof}
We use induction on $d$. The claim is trivial
for $d=1$. Let us denote by $S$ an $\alpha$-multiseparator.

Let us choose an arbitrary axis, and denote by $\mathcal{L}$
the $n^{d-1}$ parallel lines in the grid which go in that
direction. Let $\mathcal{L}'\subset \mathcal{L}$ be the set of
those lines which intersect $S$.  Note that every other element of
$\mathcal{L}$ contains vertices only from one component of $G\setminus S$. If
$|\mathcal{L}'|\ge (1-\alpha) n^{d-1}/d$, then we are done.
Hence we can suppose $|\mathcal{L}'|< (1-\alpha) n^{d-1}/d$.

Elements of $\mathcal{L}'$ cover less than $(1-\alpha) n^d/d$
points, hence for any component $C$ of $G\setminus S$, the other components together contain at least $((1-\alpha) d-(1-\alpha))n^d/d$ vertices, which are not covered by elements of
$\mathcal{L}'$. This means that there are at least $(1-\alpha)
(d-1)n^{d-1}/d$ elements of $\mathcal{L}$ which contain
only vertices not in $C$. Now consider a
hyperplane in the grid, orthogonal to the direction of the
lines of $\mathcal{L}$, and denote by $\cal H$ the vertices of $G_d(n)$ that belong to the hyperplane. Clearly, $\cal H$ contains at least
$(1-\alpha)(d-1)n^{d-1}/d$ elements not in $C$, hence $S\cap \cal H$ is an $\alpha'$-multiseparator of $\cal H$ (with $\alpha':=1-(1-\alpha)(d-1)/d$) and so we can apply induction on each of these $(d-1)$-dimensional hyperplanes.

By induction, there are at least $(1-\alpha)(d-1)n^{d-2}/d(d-1)$ elements of $S$ in every such
hyperplane, which gives at least $n(1-\alpha)n^{d-2}/d=(1-\alpha) n^{d-1}/d$ elements in total.
\end{proof}

Before proving the stronger version of this result, we need to introduce some notations and results.

Let $A$ be an arbitrary set of vertices. The set of vertices that are not in $A$, but are connected to some vertex of $A$ is called the \emph{boundary} of $A$, denoted by $\partial A$.
Following the notations of Bollob\'as and Leader \cite{bollobas}, we define an order on the vertices, the simplicial order, by setting $x < y$ if
$\sum x_i <\sum y_i$, or $\sum x_i =\sum y_i$ and for some $j$
we have $x_j > y_j$ and $x_i = y_i$ for all $i < j$.
This coincides with the lexicographic order according to the vector
$(\sum x_i,-x_1,-x_2,\dots,-x_n)$.

\begin{theorem}[Bollob\'as and Leader \cite{bollobas}]\label{bollobas}
In $G_d(n)$, among sets of vertices of a given size, the initial
segment of the simplicial order has the smallest boundary.
\end{theorem}

The special case $n=2$, i.e., the hypercube, was previously treated by Harper \cite{harper}, while the unbounded case of $n=\infty$ was solved by Wang and Wang \cite{wang}.
We note that in the paper of Bollob\'as and Leader the definition of boundary is different: they also include $A$ in $\partial A$.

We will also need some results about the volume of slices of a cube, i.e., intersections of the cube with specific hyperplanes. For a contemporary approach to this area we refer to \cite{bsplines}. In the next theorem $H^d(t)$ denotes the following set in the $d$-dimensional unit cube $I^d$: $H^d(t)=\{\,x\in I^d\mid \sum x_i=t\,\}$; $\Vol_{i}$ denotes the $i$-dimensional volume of some set of dimension $i$.

\begin{theorem}[\cite{
polya, bsplines}]\label{cubeslice}
$\lim_{d\to\infty}
\Vol_{d-1}\bigl(H^d(d/2+s\sqrt{d})\bigr)=\sqrt{\frac{6}{\pi}}e^{-6s^2}$,
for each fixed $s$.
\end{theorem}

Let $L_k$ denote the $k$-th \emph{layer} of $G_d(n)$: the
 set of all vertices in $G_d(n)$ whose coordinates sum
to $k$.
The layer range from 0 to $(n-1)d$.
We define the size of the ``middle-most'' layers
$Z_{n,d}$ by
\begin{align*}
Z_{n,d} &:=
\begin{cases}
|L_{ ((n-1)d-1)/2 }|
=|L_{ ((n-1)d+1)/2 }|,
& \text{for $(n-1)d$ odd,}\cr
\min\{|L_{ (n-1)d/2-1 }|,|L_{ (n-1)d/2 }|,|L_{ (n-1)d/2+1 }|
\},
& \text{for $(n-1)d$ even}.\cr
\end{cases}
\\
Z_{n,d}^{\max} &:=
\begin{cases}
|L_{ ((n-1)d-1)/2 }|
=|L_{ ((n-1)d+1)/2 }|=Z_{n,d},
& \text{for $(n-1)d$ odd,}\cr
|L_{ (n-1)d/2 }|, 
& \text{for $(n-1)d$ even}.\cr
\end{cases}
\end{align*}
In the even case, we actually know that the middle level
$L_{ (n-1)d/2 }$ is the largest of the three levels in the definition of
$Z_{n,d}$, as the levels decrease symmetrically in size
from the middle to the ends~\cite{bruijn}.
From discretizing the above theorem,
one can obtain the following bound on $Z_{n,d}$.
Its proof can be found in Appendix \ref{app:1}.

\begin{cor}\label{appendix1}
For every $d$, there exists a constant $C_d$ such that
\begin{align*}
 \nonumber 
  Z_{n,d}
&= C_d/\sqrt d\cdot n^{d-1}\pm O(n^{d-2}) \text{ and}
\\  Z_{n,d}^{\max}
&= C_d/\sqrt d\cdot n^{d-1}\pm O(n^{d-2}).
\end{align*}
 $C_d \to \sqrt{{6}/{\pi}}$ as $d\to\infty$.
\end{cor}

Now we are ready to prove Theorem \ref{parti}. 

\begin{proof}[Proof of Theorem \ref{parti}]
We start with the lower bound.
Let us denote by $S$ a $1/2$-biseparator which separates
the vertex set $A$ and $B$ (such that $V=A \cup B\cup S$). If $|S|\ge
Z_{n,d}$ we are done. Thus we suppose
 that $|S|<Z_{n,d}$.
 Denote by $A'$ the vertex set of size $|A|$ which is an initial
 segment of the simplicial order. By Theorem \ref{bollobas} we know
 that $|S|\ge|\partial A|\ge|\partial A'|$.

By the definition of the simplicial order, $\partial A'$ is contained
in the union of two successive layers $k$ and $k+1$:
$\partial A'=P_1\cup P_2$,
where  $P_1\subseteq L_{k  }$
and  $P_2\subseteq L_{k+1}$.
First we claim that $k$ must be very close to the middlemost layer.
More precisely, if $nd$ is odd, we can assume $k=\frac{nd-1}2$,
and if $nd$ is even, we can assume $k=\frac{nd}2-1$
or $k=\frac{nd}2$.

We treat only the odd case, the even case being similar.
First, we show that $A'$ must reach at least level $k=\frac{nd-1}2$.
If $A'$ were disjoint from $L_{k}$, we would get
$$
|A|+|S|
=
|A'|+|S|
<
|A'|+Z_{n,d} = |A' \cup L_{k}|
\le n^2/2,
$$
since the last set contains only vertices in the lower half of the
levels.
This contradicts the requirement fact that $A\cup S$ must cover at least half
of the vertices.
Secondly, if $A'$ would contain vertices of level $k+1$, it would
contain more than the levels $0,1,\dots,k$ which make up half of all
vertices.
This is again a contradiction to the
 $1/2$-biseparator property.

 By the definition of $Z_{n,d}$, we have now established that each of the
 two central layers $L_{k}$ and $L_{k+1}$ contains at least $Z_{n,d}$
 points.
To conclude the proof, we show that the separator
$\partial A'$ which is
contained in the two layers $L_{k}$ and $L_{k+1}$ must have size
at least
$Z_{n,d} -O(n^{d-2})$.
If a vertex $v=(x_1,\dots, x_d)$ of $L_{k+1}$ is not in $P_2$,
 then the adjacent vertex
$v^-$ defined by
 $v^-=(x_1,\dots x_{d-1},x_d-1)$ must be in $P_1$ unless it
 is not a point of the grid $ G(n,d)$ (i.e., $x_d=0$):
$$
(L_{k+1} \setminus P_2)^- \cap G(n,d) \subseteq P_1
$$
Since the number of vertices of $L_{k+1}$ for which $x_d=0$ is
$O(n^{d-2})$, we obtain
$$
|L_{k+1}| - |P_2| - O(n^{d-2}) \le |P_1|,
$$
from which the bound
$|\partial A'|=|P_1|+| P_2|
\ge Z_{n,d} -O(n^{d-2})$ follows.

For the upper bound, we simply take the central layer
$L_{ \lfloor(n-1)d/2\rfloor }$ 
of size $Z_{n,d}^{\max}$
as a biseparator.
\end{proof}



Now we are ready to prove Theorem~\ref{alsofeles}, that $s\bi_{1/2}(G)\le h_P(G)$.

\begin{proof}[Proof of Theorem~\ref{alsofeles}.]
We will use an adversary argument for the lower bound on the
number of queries. The
adversary will try to answer the queries in such a way that the
discovery of the endvertex by the searcher is delayed as much as possible. The
adversary need not choose a path $D$ in advance, but it is required that
the answers remain consistent with \emph{some} path.

  Let $Q$ denote the vertices that have been queried so far in the search.
	We will show that the adversary can achieve that after the other end of the path is found, $Q$ becomes a $1/2$-biseparator.
	The adversary maintains a component $C$ of $V-Q$, see Figure~\ref{strategy}.
$C$~is the set of vertices which can possibly be the endvertex of the  path.
(The adversary will follow a greedy strategy of
keeping this set as large as possible.)
%
In addition to $C$, the adversary maintains a path $P$ between
  $s$ and some vertex $p\in C$, which will be part of the final path
  and for which $P\cap C=\{p\}$.
The remaining components of $V-Q$ are partitioned into two sets
 $V\setminus (Q\cup C) = A\cup B$  such that
  both $A$ and $B$ contain at most $|V|/2$ vertices and there are no edges between $A$ and $B$.
Thus we always have a partition into four disjoint sets
$V = Q \cup A \cup B \cup C$.
  The adversary can reveal all these data 
to the searcher as free additional information. 
  Initially, $C=V$, $p=s$ and $Q=A=B=\emptyset$.

\begin{figure}[htb]
  \centering
  \includegraphics{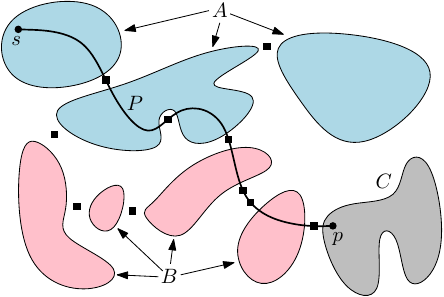}
  \caption{A schematic drawing of the situation maintained by the adversary.
The queried vertices, $Q$, are marked by squares.}
  \label{strategy}
\end{figure}

  The strategy is the following. 
If the queried vertex $q$ is in $Q$,
the adversary repeats the previous answer for this vertex.
	If $q\in P\setminus\{p\}$, the adversary answers by reporting the ingoing and outgoing
  edge of $P$ at that vertex.  If $q\notin C \cup P$, then the answer is that ``the path does
  not pass through this vertex.'' In these cases, no new information
  is revealed to the searcher. The vertex $p$, the set $C$, and the path
  $P$ remain unchanged; the only change is that $q$ is moved from $A \cup B$ to $Q$.

  Let us now look at the case $q\in C$.
  Let $C\setminus \{q\} = D_{1}\cup D_{2}\cup \dots\cup D_{m}$ be the
  partition of  $C\setminus \{q\}$ into $m\ge 1$ connected components.
The adversary chooses a largest component $D_{j}$, and
  will answer in such a way that the new set $C$ becomes
$C\new=D_{j}$.

\begin{figure}
  \centering
  \includegraphics[scale=1.1]{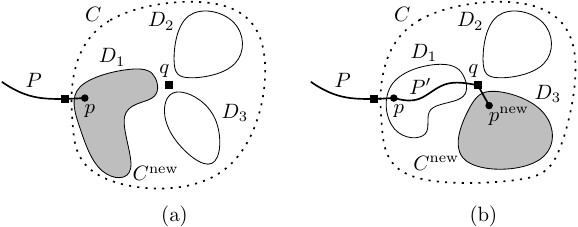}
  \caption{Updating the set $C$ after a query $q$}
  \label{response}
\end{figure}

 Therefore, if
  $C\new $ contains $p$, the answer is again ``the path does
  not pass through this vertex,'' see Figure~\ref{response}a.
  The current endpoint $p$ and the path
  $P$ are unchanged.
  If $C\new $ does not contain $p$ (including the case $q=p$), then choose $p\new \in C\new $
  to be a neighbor of $q$, see Figure~\ref{response}b. As $q$ was a possible endpoint of the path
  before this step, there is a path $P\new$  from $p$ to $q$
  which lies in $C\setminus C\new $. The adversary uses $P\new$ and
  the edge $qp\new $ to extend the path $P$ to a longer path
  $P\new$. (This is the only case when the path is updated.)
  The adversary reports the last arc of $P\new$ as the ingoing arc at
  $q$ and $qp\new $ as the outgoing arc.

To maintain the invariant that $|A|,|B|\le |V|/2$, we go through
the components
  $D_i\ne C\new$ and  add them either to $A$ or to $B$ (to eventually obtain $A\new $ and $B\new $), whichever is smaller.
If, for example, $|A|\le |B|$, then $|A|+|D_i|\le |B|+|C\new|\le |V|/2$
as $A,D_i,
 B, C\new$ are disjoint subsets of $V$.
Therefore, the invariant is maintained.

  The searcher can only identify $t$, the end of the path, when
 $|C|$ becomes~$1$.
By assumption, the graph $G$ has at least two vertices and is
connected,
and therefore $Q\ne \emptyset$. Thus, 
at this point,
$$\min\{|A|,|B|\} \le |V\setminus (Q\cup C)|/2\le (|V|-1-1)/2
= |V|/2-1.$$
We can now add the singleton set $C=\{t\}$ to the smaller of $A$ and $B$
without exceeding the size bound $|V|/2$.
The set $Q$ of queried vertices forms thus a $1/2$-biseparator.
\end{proof}

\begin{cor}
$h_P(G_d(n))=\Omega( n^{d-1}/\sqrt{d})$.
\qed
\end{cor}

Theorem \ref{racs} summarizes the above results. The lower and
upper bounds are quite close. Specifically, if we consider $d$ as fixed, then the theorem gives exact asymptotics in $n$ for the
needed number of queries.

\section{Concluding Remarks}
\label{sec:conclusion}

Here we mention three more variants of the problem.

In the first variant, we consider any directed subgraph of $G'$ and a vertex $s$ with larger out-degree than in-degree. In this version there is a vertex with higher in-degree than out-degree, our goal is to find such a vertex. All of our algorithms work in this case, and obviously the same lower bounds hold.

In the second variant, $D$ consists of directed paths and cycles, but we also assume that they cover every vertex. This is a special case of our model, hence the upper bounds hold. However, a lower bound similar to Theorem \ref{alsofeles} is not plausible, as there are graphs that have only big separators, yet there are only a few valid choices for $D$. For example if $G$ contains a vertex of degree one, different from the source, then this vertex must be the endvertex. But in case of grid graphs we can show that the additional assumption on $D$ does not make the problem much easier.

Denote by $h_U(G)$ the minimum number of queries
needed to find an endvertex in the worst-case for any $s\in G$.
Now we show how to give a lower bound for $h_U(G_d(n))$.
Let us suppose we are given an $r_1 \times r_2 \times r_3 \times
\dots \times r_d$ grid graph $G$. Then let $G^{4,4}$ denote the
$4r_1 \times 4r_2 \times r_3 \times \dots \times r_d$ grid graph.

\begin{theorem}\label{kapcs} Let $G$ be a grid graph. Then $h_P(G) \le h_U(G^{4,4})$.
\end{theorem}

The proof of this theorem can be found in Appendix \ref{app:2}.

One can easily see that if 4 divides $n$ and $G$ is the $n/4
\times n/4 \times n \times \dots \times n$ grid graph, then
$G_d(n)=G^{4,4}$. We need a lower bound on the size of
separators in $G$. It is easy to see that if we replace every vertex of $G$ by 16 vertices to get $G_d(n)$, an $\alpha$-separator is replaced by an $\alpha$-separator, hence the same lower bound of $\Omega(n^{d-1}/\sqrt{d})$, divided by 16, holds for $G$.

\begin{cor} $ \Omega (n^{d-1}/\sqrt{d}) \le h_U(G_d(n)) \le O(n^{d-1})$.
\end{cor}

In the third variant, $D$ is undirected.
Our goal is to find another endvertex and the answer to the query is the at most two incident edges.
Obviously, this is a harder problem than the directed variant.
Hence our lower bounds hold, and one can easily modify our proofs to get the same upper bounds as well.
For example, in Observation~\ref{cut}, the endvertex is in the component $Y_i$ which is connected to $S$ by an odd number of edges, counting an extra edge for the component of $s$.

Finally, a straightforward application of our proofs gives the asymptotics to a question recently asked on MathOverflow \cite{MO}, which is the following.
Given a path $P_1$ from the bottom-left vertex of an $n \times n$ grid to its top-right vertex, and another path $P_2$ from its top-left vertex to its bottom-right vertex, how many queries are needed to find a vertex contained in both paths?
The proofs of Theorems~\ref{alsofeles} and \ref{th_cycle} can be adapted to show that $\Theta(n)$ queries are necessary and sufficient.

\subsection*{Acknowledgment}
We would like to thank our anonymous referee for the remarks that improved the presentation of the paper.

\appendix

\section{Biseparators for Ternary Trees}
\label{ternary}
We show that a rooted ternary tree with $k+1$ complete levels has
 $s\bi _{1/2}(G)=\Theta(k)$.
Any root-to-leaf path is a $1/2$-biseparator, establishing the
upper bound.
Let us turn to the lower bound.
A complete ternary tree of height $h$ has $n=(3^{h+1}-1)/2$ vertices.
It is convenient to give each vertex a ``weight'' of 2. The total
weight of the tree becomes $2n=3^{k+1}-1$, which is very near to
a power of 3. In ternary notation, $2n=(22\ldots 2)_3$ with $k$ twos,
and the ideal weight for the halves of the biseparator is
$2n/2=n=(11\ldots 1)_3$.

After removing a separating set, 
any union of components of
the complement can be represented as a sum and difference of subtrees.
Here, by a subtree we mean a node together with all its descendents.
If the separator has $s$ nodes, we must be able to group the resulting
components into a
 set that has between
$n/2-s$ and $n/2$ nodes, i.e., weight
between $n-2s$ and $n$.
Each separator node creates at most four new
 subtrees from which
the sum and difference can be formed: its own 
 subtree and the three children subtrees.
(These latter ones exist only if 
the node was not a leaf.)
So with $s$ separating nodes, we get $1+4s$ subtrees from which to
form the sum and difference.
Each tree has a weight of the form $3^h-1$.

 If we take a sum and difference of $L\le4s+1$ subtrees
we must fulfill the inequality
$$n-2s
\le
\sum_{i=1}^L (\pm (3^{h_i} -1))
\le n,$$
which implies
$$n-2s-L
\le
\sum_{i=1}^L (\pm 3^{h_i})
\le n+L$$
and
$$n-6s-1
\le
\sum_{i=1}^L (\pm 3^{h_i})
\le n+4s+1.$$
For any number $p$ in the range
$n-6s-1
\le p
\le n+4s+1$,
the ternary representation
starts with at least $k- 1-\lceil\log_3(6s+1)\rceil$ ones.
On the other hand, one easily sees by induction that a
sum and difference of $L$ powers of~3 has at most $L$ ones in its
ternary representation.
We thus get the relation
$4s+1\ge L \ge k- 1-\lceil\log_3(6s+1)\rceil$, from which
$s\ge \Omega(k)$ follows.
\qed

\section{Proof of Corollary \ref{appendix1}}\label{app:1}

We show that for any fixed $\delta\ge 0$ (and then by symmetry for every $\delta<0$ too), whenever $(n-1)d/2+\delta$ is an integer,
\begin{equation}
 \nonumber 
  |L_{(n-1)d/2+\delta}|
= C_d/\sqrt d\cdot n^{d-1}\pm O(n^{d-2}).
\end{equation}

We define
$C_d = \Vol_{d-1}H^d(d/2)$, i.e., the volume of the middle slice of the unit hypercube.
Setting $s=0$ in Theorem \ref{cubeslice} establishes
the convergence of $C_d$ to $\sqrt{{6}/{\pi}}$.

  The layer $L_k$, for $k=(n-1)d/2+\delta$, is a discrete version of a slice of a cube.
	If we fix the first $d-1$ coordinates, then there is at most one vertex in
  $L_k$ that has these first $d-1$ coordinates.
	Thus $|L_k|=|L_k'|$, where $L_k'$ is the projection of $L_k$ along the last axis.
  
To estimate the size of $L_k'$ (and thus of $L_k$) take first the middle slice $H^d(d/2)$ of the continuous unit cube and project it to the first $d-1$ coordinates, yielding the polytope $H^d(d/2)'$.
As the normal vector of the slice is $(1,1,\dots,1)$, projecting it to
the hyperplane orthogonal to the last axis scales the volume by a
factor of $1/\sqrt d$:
$$
\Vol_{d-1}H^d(d/2)'
=
\Vol_{d-1}H^d(d/2)/ \sqrt d.$$
  
Now let $H^d(d/2)''=nH^d(d/2)'$, i.e., we blow up $H^d(d/2)'$ by a factor $n$.
Let $M$ be the set of grid points in this $H^d(d/2)''$.
As for fixed $d$, $H^d(d/2)''$ is a factor-$n$ blow up of some fixed $(d-1)$-dimensional convex polytope, the difference between its volume and the number of grid points in it is $O(n^{d-2})$ (this follows basically from the definition of the volume, for details see e.g., \cite[Proposition 4.6.13]{ec1}), thus 

$$|M|=n^{d-1}\Vol_{d-1}H^d(d/2)'+O(n^{d-2})=$$
$$=n^{d-1}\Vol_{d-1}H^d(d/2)/\sqrt{d}+O(n^{d-2})=C_d/\sqrt d\cdot n^{d-1}+ O(n^{d-2}).$$

Now we are left to show that $|L'_k|=|M|+O(n^{d-2})$. For that it is enough to show that $|L'_k \setminus M|$ and $|M \setminus L'_k|$ are $O(n^{d-2})$. For all of these points the sum of the $d-1$ coordinates is equal to $(n-1)d/2+i$ (resp. $(n-1)d/2-n+i$) for some $0< i\le \delta$. This is $O(n^{d-2})$ points for every $i$, altogether $2\delta O(n^{d-2})=O(n^{d-2})$ points, which finishes the proof.
\qed

\section{Proof of Theorem \ref{kapcs}}\label{app:2}

Suppose we are given a  grid graph $G$ and an Algorithm A which finds $t$ in $G^{4,4}$ in case one path and some cycles cover every vertex. We show an Algorithm B which finds the endvertex in $G$ in case there is only a directed path. We can naturally identify
every vertex of $G$ with a $4 \times 4$ grid in $G^{4,4}$:
the vertex $v=(i_1,\dots i_d)$ corresponds to the axis-parallel
$4\times4$ rectangle (we call it a block) $B(v)$ having $16$
vertices, whose two opposite corners are $(4i_1-3,4i_2-3,i_3, \dots
i_d)$ and $(4i_1,4i_2,i_3, \dots i_d)$.
We call $(4i_1-3,4i_2-3,i_3,
\dots i_d)$ and $(4i_1,4i_2,i_3, \dots i_d)$ the \emph{even} corners
and the two other corners
$(4i_1-3,4i_2,i_3, \dots i_d)$ and $(4i_1,4i_2-3,i_3, \dots i_d)$
the \emph{odd} corners.

Consider a directed path $P$ in $G$. We call a system of a
directed path and some directed cycles in $G^{4,4}$ \emph{good}
if they cover every vertex and the path goes through
exactly those blocks which correspond to the vertices of $P$, in
the same order.

Now we construct good systems. If a vertex $v \in V(G)$ is not on
the path, we cover the corresponding block by a cycle. In case of
a vertex $v=(i_1,\dots ,i_d)$ on the path in $G$, the directed path
arrives at the corresponding block $B(v)$ in some corner $p_1(v)$,
and goes straight to a neighboring corner $p_2(v)$, where it
leaves. The remaining vertices form a $4 \times 3$ rectangle, which can
be covered by a cycle. Finally, when $v$ is the very last vertex on the path, we
define $p_1(v)$ similarly, and cover the remaining vertices by a
path starting in $p_1(v)$. 

Our good systems will satisfy an additional property. If, for a vertex
$v=(i_1,\dots i_d) \in G$, the coordinate sum $\sum_{j=3}^d i_j$ is even, then the first vertex $p_1(v)$ of the path in the corresponding block is an even corner, and the last vertex $p_2(v)$ is an odd corner. In case $\sum_{j=3}^d i_j$ is odd, it is the other way round. Note that if it is true for $B(s)$, it has to be true for every other block as well. Indeed, when the path leaves a block at, for example, an odd corner, it either moves in one of the first two dimensions (then it arrives to an even corner, and $\sum_{j=3}^d i_j$ does not change), or in another dimension (then it arrives to an odd corner, but the parity of $\sum_{j=3}^d i_j$ changes).

Note that these properties do not
uniquely determine the system. We will incrementally determine the
graph as queries arrive.

Now we are ready to define Algorithm B. At every step we call Algorithm
A, and then answer such a way that at the end we get a good
system. If Algorithm A would query a vertex $v$ in $G^{4,4}$,
Algorithm B queries the corresponding vertex $v'$ in $G$ instead (i.e.,
the vertex $v'$ with $v\in B(v')$). Using the answer for this query, we choose all the edges incident to vertices of $B(v')$ and answer to Algorithm A according to this.
If $v'$ has been asked before, we have already
determined the edges in $B(v')$, and answer accordingly.
Suppose that $v'$ has not been queried before.
 In
case the answer is that $v'$ is not on the path, choose an
arbitrary cycle covering the vertices of the corresponding block $B(v')$ and answer
according to the edges incident to $v$.

In case the answer gives two arcs $uv'$ and $v'w$,
we have to choose the entering vertex
$p_1(v')$ and the exit vertex $p_2(v')$.
We will discuss this choice below.
This choice will define 5 edges
on the path and a cycle of length 12.
 One edge connects the
blocks corresponding to $u$ and $v$, leaving the last vertex of
the path in $B(u)$ and arriving at the first vertex of the path in
$B(v')$, i.e., this edge is $p_2(u)p_1(v')$. Similarly we add the
edge $p_2(v')p_1(w)$. We also add the three edges which connect
$p_1(v')$ and $p_2(v')$. Finally we cover the remaining 12 vertices
with a cycle.

We still have to tell which one of the two possible first vertices
we use as $p_1(v')$, and similarly for the possible last vertices.
If $p_2(u)$ has already been determined, this fixes the choice
of $p_1(v')$ as the vertex adjacent to it.
If $uv'$ is parallel to one of the first two axes, this
also reduces the choice of the corner $p_1(v')$
to one possibility.
 Otherwise we pick $p_1(v')$ arbitrarily among the two choices.
The exiting vertex
$p_2(v')$ is determined analogously.

 Even if Algorithm A would know all answers
in $B(v')$, it does not give more information than what Algorithm B
knows after asking $v'$. Algorithm A does not finish before
Algorithm B finds the end vertex, thus Algorithm A needs at least as many queries as Algorithm B (on the respective graphs), which finishes the proof.
\qed

\end{document}